\documentclass[12pt,letter]{amsart}
\usepackage{xcolor}
\usepackage[all,color]{xy}
\usepackage{amsmath,amssymb,amsthm,amscd,amsxtra,amsfonts,mathrsfs,graphicx,enumerate,bm,slashed,array,setspace,amsfonts}
\input xy
\xyoption{all}
\newtheorem{Theorem}{Theorem}[section]
\newtheorem{Lemma}[Theorem]{Lemma}
\newtheorem{Proposition}[Theorem]{Proposition}

\newtheorem{Remark}[Theorem]{Remark}

\newtheorem{Definition}[Theorem]{Definition}

\newtheorem{Notation}[Theorem]{Notation}
\newtheorem*{Theorem A}{Theorem A}
\newtheorem*{Theorem B}{Theorem B}

\newcommand*{\overbar}[1]{\mkern 1.5mu\overline{\mkern-1.5mu#1\mkern-1.5mu}\mkern 1.5mu}
\advance\evensidemargin-.5in
\advance\oddsidemargin-.5in
\advance\textwidth1in

\begin{document}
\author{Karin Baur}
\address{School of Mathematics, University of Leeds, Leeds, LS2 9JT, United Kingdom}
\address{On leave from the University of Graz}
\email{pmtkb@leeds.ac.uk}
\author{Charlie Beil} 
\address{Institut f\"ur Mathematik und Wissenschaftliches Rechnen, Universit\"at Graz, Heinrichstrasse 36, 8010 Graz, Austria.}
 \email{charles.beil@uni-graz.at}
\title[Examples of geodesic ghor algebras on hyperbolic surfaces]{Examples of geodesic ghor algebras\\ on hyperbolic surfaces}
 \keywords{Dimer algebra, hyperbolic surface, non-noetherian ring, noncommutative algebraic geometry.}
 \subjclass[2010]{13C15, 14A20}
 \date{}

\maketitle

\begin{abstract}
Cancellative dimer algebras on a torus have many nice algebraic and homological properties. 
However, these nice properties disappear for dimer algebras on higher genus surfaces. 
We consider a new class of quiver algebras on surfaces, called `geodesic ghor algebras', that reduce to cancellative dimer algebras on a torus, yet continue to have nice properties on higher genus surfaces. 
These algebras exhibit a rich interplay between their central geometry and the topology of the surface.
We show that (nontrivial) geodesic ghor algebras do in fact exist, and give explicit descriptions of their central geometry.
This article serves a companion to the article `A generalization of cancellative dimer aglebras to hyperbolic surfaces', where the main statement is proven. 
\end{abstract}

\section{Introduction}

Cancellative dimer algebras on a torus form a prominent class of noncommutative crepant resolutions and Calabi-Yau algebras, e.g., \cite{Br,D,B2}.  
In particular, they are homologically homogeneous endomorphism rings of modules over their centers, and their centers are $3$-dimensional toric Gorenstein coordinate rings.
These interesting properties vanish, however, once dimer algebras are placed on higher genus surfaces. 
For example, the center of a dimer algebra on a surface of genus $g \geq 2$ is simply the polynomial ring in one variable, and so there are no interesting interactions between the topology of the surface and the algebras central geometry. 

In this article we consider special quotients of dimer algebras, called `ghor algebras', whose central geometries are, in contrast, closely related to the topology of the surface. 
A ghor algebra is a quiver algebra whose quiver embeds in a surface, and with relations determined by the perfect matchings of its quiver (the precise definition is given in Section \ref{results}).
Ghor algebras were introduced in \cite{B1,B4} to study nonnoetherian dimer algebras on a torus.\footnote{Ghor algebras were originally called `homotopy algebras' (e.g., in \cite{B4}) because their relations are homotopy relations on the paths in the quiver when the surface is a torus.
However, in the higher genus case their relations also identify homologous cycles, and therefore the name `homotopy' is less suitable.  The word `ghor' is Klingon for surface.}
In Section \ref{results} we introduce a special property that certain ghor algebras possess, called `geodesic'.
On a torus, a ghor algebra is geodesic if and only if it is a cancellative dimer algebra (if and only if it is a noncommutative crepant resolution, if and only if it is noetherian \cite{D,B2}).  
On higher genus surfaces, geodesic ghor algebras remain endomorphism rings of modules over their centers, but new features arise.
The purpose of this article is to present explicit examples of ghor algebras that exhibit these new features.

In each example, we will describe the `cycle algebra', which is a commutative algebra formed from the cycles in the quiver. 
The cycle algebra and center coincide for geodesic ghor algebras on a torus, but may differ on higher genus surfaces.
However, the two commutative algebras remain closely related: if the center is nonnoetherian, then the cycle algebra is a depiction of the center. 
The cycle algebra thus plays a fundamental role in describing the geometry of the center when the center is nonnoetherian. 

We will also show that polynomial rings in at least three variables are geodesic ghor algebras.
The results stated in Section \ref{results} are proven in the companion article \cite{BB}.

\section{Ghor algebras: background and main results} \label{results}

\begin{Notation} \rm{
Throughout, $k$ is an uncountable algebraically closed field.
We denote by $\operatorname{Spec}S$ and $\operatorname{Max}S$ the prime ideal spectrum (or scheme) and maximal ideal spectrum (or affine variety) of $S$, respectively.
Given a quiver $Q$, we denote by $kQ$ the path algebra of $Q$; by $Q_{\ell}$ the paths of length $\ell$; by $\operatorname{t}, \operatorname{h}: Q_1 \to Q_0$ the tail and head maps; and by $e_i$ the idempotent at vertex $i \in Q_0$.
By \textit{cyclic subpath} of a path $p$, we mean a subpath of $p$ that is a nontrivial cycle.
}\end{Notation}

In this article will consider surfaces $\Sigma$ that are obtained from a regular $2N$-gon $P$, $N \geq 2$, by identifying the opposite sides of $P$.
This class of surfaces includes all smooth orientable compact closed connected genus $g \geq 1$ surfaces.
Specifically,
\begin{itemize}
 \item if $P$ is a $4g$-gon, then $\Sigma$ is a smooth genus $g$ surface; and 
 \item if $P$ is a $2(2g+1)$-gon, then $\Sigma$ is a genus $g$ surface with a pinched point.
\end{itemize}
The polygon $P$ is then a fundamental polygon for $\Sigma$.

If $N = 2$, then $\Sigma$ is a torus, and the covering space of $\Sigma$ is the plane $\mathbb{R}^2$.
For $N \geq 3$, the covering space of $\Sigma$ is the hyperbolic plane $\mathbb{H}^2$.
The hyperbolic plane may be represented by the interior of the unit disc in $\mathbb{R}^2$, where straight lines in $\mathbb{H}^2$ are segments of circles that meet the boundary of the disc orthogonally.
In the covering, the hyperbolic plane is tiled with regular $2N$-gons, with $2N$ such polygons meeting at each vertex.
In this case, $\Sigma$ is said to be a hyperbolic surface.

\begin{Definition} \label{definitions} \rm{ \ \\
\indent $\bullet$ A \textit{dimer quiver} on $\Sigma$ is a quiver $Q$ whose underlying graph $\overbar{Q}$ embeds in $\Sigma$, such that each connected component of $\Sigma \setminus \overbar{Q}$ is simply connected and bounded by an oriented cycle, called a \textit{unit cycle}.

\begin{itemize}
 \item[-] A \textit{perfect matching} of a dimer quiver $Q$ is a set of arrows $x \subset Q_1$ such that each unit cycle contains precisely one arrow in $x$.
Throughout, we assume that each arrow is contained in at least one perfect matching.
 \item[-] A perfect matching $x$ is called \textit{simple} if $Q \setminus x$ contains a cycle that passes through each vertex of $Q$ (equivalently, $Q \setminus x$ supports a simple $kQ$-module of dimension vector $(1,1, \ldots, 1)$).
\end{itemize}

We denote by $\mathcal{P}$ and $\mathcal{S}$ the set of perfect and simple matchings of $Q$, respectively.
We will consider the polynomial rings generated by these matchings, $k[\mathcal{P}]$ and $k[\mathcal{S}]$. 

$\bullet$ Denote by $e_{ij} \in M_n(k)$ the $n \times n$ matrix with a $1$ in the $ij$-th slot and zeros elsewhere.
Consider the two algebra homomorphisms
\begin{equation*} \label{eta}
\eta: kQ \to M_{|Q_0|}\left(k[\mathcal{P}]\right) \ \ \ \ \text{ and } \ \ \ \ \tau: kQ \to M_{|Q_0|}\left(k[\mathcal{S}]\right)
\end{equation*}
defined on the vertices $i \in Q_0$ and arrows $a \in Q_1$ by
$$\begin{array}{c} 
\eta(e_i) = e_{ii}, \ \ \ \ \ \ \ \ \eta(a) = e_{\operatorname{h}(a),\operatorname{t}(a)} \prod_{\substack{x \in \mathcal{P} :\\ x \ni a}} x,\\
\tau(e_i) = e_{ii}, \ \ \ \ \ \ \ \ \tau(a) = e_{\operatorname{h}(a),\operatorname{t}(a)} \prod_{\substack{x \in \mathcal{S} :\\ x \ni a}} x,
\end{array}$$
and extended multiplicatively and $k$-linearly to $kQ$.
We call the quotient 
$$A := kQ/\operatorname{ker}\eta$$
the \textit{ghor algebra} of $Q$.

$\bullet$ The \textit{dimer algebra} of $Q$ is the quotient of $kQ$ by the ideal 
$$I = \left\langle p - q \ | \ \exists a \in Q_1 \text{ such that } pa, qa \text{ are unit cycles} \right\rangle \subset kQ,$$
where $p,q$ are paths.
}\end{Definition}

A ghor algebra $A = kQ/\operatorname{ker}\eta$ is the quotient of the dimer algebra $kQ/I$ since $I \subseteq \operatorname{ker}\eta$: if $pa, qa$ are unit cycles with $a \in Q_1$, then
$$\eta(p) = e_{\operatorname{h}(p),\operatorname{t}(p)} \prod_{\substack{x \in \mathcal{P}: \\ x \not \ni a}} x = \eta(q).$$
Dimer algebras on non-torus surfaces have been considered in the context of, for example, cluster categories \cite{BKM,K}, Belyi maps \cite{BGH}, and gauge theories \cite{FGU,FH}.

\begin{Notation} \rm{
Let $\pi: \Sigma^+ \to \Sigma$ be the projection from the covering space $\Sigma^+$ (here, $\mathbb{R}^2$ or $\mathbb{H}^2$) to the surface $\Sigma$.
Denote by $Q^+ := \pi^{-1}(Q) \subset \Sigma^+$ the (infinite) covering quiver of $Q$.
}\end{Notation}

We introduce the following special class of ghor algebras that generalizes cancellative dimer algebras on a torus.
\newpage
\begin{Definition} \rm{ 
Given a cycle $c$, we define the \textit{class} of $c$ to be
$$[c] := \sum_{k \in [N]} (n_k - n_{k+N}) (\delta_{k \ell})_{\ell} \in \mathbb{Z}^N,$$ 
where for $k \in [2N]$, $c$ transversely intersects side $k$ of $P$ $n_k$ times.
If $\Sigma$ is smooth (that is, $N$ is even), then $[c]$ is the homology class of $c$ in $H_1(\Sigma) := H_1(\Sigma, \mathbb{Z})$. 
\begin{itemize}
 \item A cycle $p \in A$ is \textit{geodesic} if the lift to $Q^+$ of each cyclic permutation of each representative of $p$ does not have a cyclic subpath.  
 \item Two cycles are \textit{parallel} if they do not transversely intersect.
 \item A ghor algebra is \textit{geodesic} if for each $k \in [2N]$ there is a geodesic cycle $\gamma_k$ with class
$$[\gamma_k] =  (\delta_{k \ell} - \delta_{k+N,\ell})_{\ell \in [N]} \in \mathbb{Z}^N,$$
with indices modulo $2N$, and a set of pairwise parallel geodesic cycles $c_i \in e_i kQ e_i$, $i \in Q_0$, such that $c_{\operatorname{t}(\gamma_k)} = \gamma_k$.
\end{itemize}
}\end{Definition}

\begin{Proposition} \label{suffices} \cite[Corollary 3.15]{BB}
If a ghor algebra $A$ is geodesic, then 
$$A := kQ/\operatorname{ker}\eta \cong kQ/\operatorname{ker}\tau.$$
In particular, it suffices to only consider the simple matchings of $Q$ to determine the relations of $A$.
\end{Proposition}

The algebra homomorphisms $\eta$ and $\tau$ on $kQ$ induce algebra homomorphisms on $A$,
$$\eta: A \to M_{|Q_0|}\left(k[\mathcal{P}]\right) \ \ \ \ \text{ and } \ \ \ \ \tau: A \to M_{|Q_0|}\left(k[\mathcal{S}]\right).$$
For $i,j \in Q_0$, consider the $k$-linear maps
$$\bar{\eta}: e_jAe_i \to k[\mathcal{P}] \ \ \ \ \text{ and } \ \ \ \ \bar{\tau}: e_jAe_i \to k[\mathcal{S}]$$
defined by sending $p \in e_jAe_i$ to the single nonzero matrix entry of $\eta(p)$ and $\tau(p)$ respectively; that is,
$$\eta(p) = \bar{\eta}(p)e_{ji} \ \ \ \ \text{ and } \ \ \ \ \tau(p) = \bar{\tau}(p)e_{ji}.$$
Then 
$$\bar{\eta}(p)|_{\substack{x = 1: \\ x \not \in \mathcal{S}}} = \bar{\tau}(p).$$
These maps are multiplicative on the paths of $Q$.
Moreover, for $a \in Q_1$, we have $a \in x$ if and only if $x \mid \bar{\eta}(a)$.

By the definition of $A$, if two paths $p,q \in e_jAe_i$ satisfy $\bar{\eta}(p) = \bar{\eta}(q)$, then $p = q$.
Furthermore, if $A$ is geodesic, then $\bar{\tau}(p) = \bar{\tau}(q)$ implies $p = q$, by Proposition \ref{suffices}. 
An important monomial is the $\bar{\tau}$-image of each unit cycle in $Q$, namely
$$\sigma := \prod_{x \in \mathcal{S}} x.$$

The following lemma will be useful in the next section.

\begin{Lemma} \label{geo}
If $p \in A$ is a cycle satisfying $\sigma \nmid \bar{\tau}(p)$, then $p$ is a geodesic cycle.
\end{Lemma}

\begin{proof}
If $p$ is not a geodesic cycle, then there is a cyclic permutation of a lift of a representative of $p$ with a cyclic subpath $q$ in $Q^+$.
Since $q$ is a cycle in $Q^+$, we have $\bar{\tau}(q) = \sigma^{\ell}$ for some $\ell \geq 1$ by \cite[Lemma 3.1.i]{BB}.
\end{proof}

\begin{Remark} \rm{
Let $A$ be a geodesic ghor algebra on a surface $\Sigma$, and fix $i \in Q_0$.
If $\Sigma$ is a torus, then the geodesic assumption implies that for each $j,k \in \pi^{-1}(i)$, there is a path $p^+$ from $j$ to $k$ in $Q^+$ such that $\sigma \nmid \overbar{\tau}(p)$ \cite[Proposition 4.20.iii]{B1}.
However, if $\Sigma$ is hyperbolic, then this implication no longer holds.

Indeed, suppose $\Sigma$ is hyperbolic, that is, $N \geq 3$.
Fix unit vectors $u_1, u_2 \in \mathbb{Z}^N$ for which $u_1 \not = \pm u_2$.
Let $p_1, p_2, q_1, q_2$ be paths in $kQ^+e_j$ that are lifts of cycles with classes
$$[\pi(p_1)] = -[\pi(p_2)] = u_1 \ \ \ \ \text{ and } \ \ \ \ [\pi(q_1)] = -[\pi(q_2)] = u_2.$$
Consider the paths
$$s := q_2p_2q_1p_1 \ \ \ \ \text{ and } \ \ \ \ t := q_2q_1p_2p_1$$
in $Q^+$.
Then $\bar{\tau}(s) = \bar{\tau}(t)$, and so $\pi(s) = \pi(t)$ in $A$.
Observe that $t$ is a cycle in $Q^+$, whereas $s$ is not.
Furthermore, the $\bar{\tau}$-image of any cycle in $Q^+$ is a power of $\sigma$ \cite[Lemma 3.1.i]{BB}.
Thus, although $s$ is not a cycle, we have $\bar{\tau}(s) = \bar{\tau}(t) = \sigma^{\ell}$ for some $\ell \geq 1$. 
It therefore follows that the monomial of each path in $Q^+$ from $\operatorname{t}(s)$ to $\operatorname{h}(s)$ is a power of $\sigma$ \cite[Lemma 3.1.ii]{BB}, a feature that never occurs if $\Sigma$ is flat. 
}\end{Remark}

If $\Sigma$ is a torus, then a ghor algebra $A$ is geodesic if and only if it is noetherian, if and only if its center $R$ is noetherian, if and only if $A$ is a finitely generated $R$-module \cite[Theorem 1.1]{B2}.
If $\Sigma$ is hyperbolic, then only one direction of the implication survives:

\begin{Proposition} \cite[Lemma 4.4]{BB}
If the center $R$ of a ghor algebra $A$ is noetherian, then
\begin{enumerate}
 \item $A$ is geodesic;
 \item $A$ is noetherian; and
 \item $A$ is a finitely generated $R$-module.
\end{enumerate}
\end{Proposition}

In contrast to the torus case, the centers of geodesic ghor algebras on hyperbolic surfaces are almost always nonnoetherian.\footnote{It is possible that there is only a finite number of geodesic ghor algebras that are noetherian for each genus $g \geq 2$.} 
We can nevertheless view such a center as the coordinate ring on a geometric space, using the framework of nonnoetherian geometry introduced in \cite{B3} (see also \cite{B5}).
In short, the geometry of a nonnoetherian coordinate ring of finite Krull dimension looks just like a finite type algebraic variety, except that it has some positive dimensional closed points. 

\begin{Definition} \rm{
A \textit{depiction} of an integral domain $k$-algebra $R$ is a finitely generated overring $S$ such that the morphism
$$\operatorname{Spec}S \to \operatorname{Spec}R, \ \ \ \mathfrak{q} \mapsto \mathfrak{q} \cap R,$$
 is surjective, and
\begin{equation*} \label{U}
U_{S/R} := \{ \mathfrak{n} \in \operatorname{Max}S \ | \ R_{\mathfrak{n} \cap R } = S_{\mathfrak{n}} \} = \{ \mathfrak{n} \in \operatorname{Max}S \ | \ R_{\mathfrak{n} \cap R} \text{ is noetherian} \} \not = \emptyset.
\end{equation*}
}\end{Definition}

For example, the algebra $S = k[x,y]$ is a depiction of its nonnoetherian subalgebra $R = k + xS$.
We thus view $\operatorname{Max}R$ as the variety $\operatorname{Max}S = \mathbb{A}^2_k$, except that the line $\{ x = 0 \}$ is identified as a $1$-dimensional (closed) point of $\operatorname{Max}R$.
In particular, the complement $\{x \not = 0\} \subset \mathbb{A}^2_k$ is the `noetherian locus' $U_{S/(k+xS)}$ \cite[Proposition 2.8]{B3}.

The following is the main theorem of the companion article \cite{BB}.

\begin{Theorem} \label{maintheorem}
Suppose $A = kQ/\operatorname{ker}\eta$ is a geodesic ghor algebra on a surface $\Sigma$ obtained from a regular $2N$-gon $P$ by identifying the opposite sides, and vertices, of $P$.
Set 
\begin{equation*} \label{R and S}
R = k\left[ \cap_{i \in Q_0} \bar{\tau}(e_iAe_i) \right] \ \ \ \ \text{ and } \ \ \ \ S = k\left[ \cup_{i \in Q_0} \bar{\tau}(e_iAe_i) \right];
\end{equation*}
then $R$ is isomorphic to the center of $A$.
Furthermore, the following holds.
\begin{enumerate}
\item If there is a cycle $p$ such that $\overbar{p}^n \not \in R$ for each $n \geq 1$, then $A$ and $R$ are nonnoetherian.
In this case, $R$ is depicted by the cycle algebra $S$.
\item The center $R$ and cycle algebra $S$ have Krull dimension
\begin{equation*}
\operatorname{dim}R = \operatorname{dim}S = N + 1.
\end{equation*}
In particular, if $\Sigma$ is a smooth genus $g \geq 0$ surface, then
$$\operatorname{dim}R = \operatorname{rank}H_1(\Sigma) + 1 = 2g + 1.$$
 \item At each point $\mathfrak{m} \in \operatorname{Max}R$ for which the localization $R_{\mathfrak{m}}$ is noetherian, the localization $A_{\mathfrak{m}} := A \otimes_R R_{\mathfrak{m}}$ is an endomorphism ring over its center: for each $i \in Q_0$, we have
$$A_{\mathfrak{m}} \cong \operatorname{End}_{R_{\mathfrak{m}}}(A_{\mathfrak{m}}e_i).$$
The locus of such points lifts to the open dense subset $U_{S/R} \subset \operatorname{Max}S$ of the cycle algebra.  
\end{enumerate}
\end{Theorem}

\section{Examples}

In the following, we consider explicit examples of geodesic ghor algebras.
Recall that $\Sigma$ is a closed surface obtained from a regular $2N$-gon $P$ by identifying the opposite sides, and vertices, of $P$. 
Set $[m] := \{1, \ldots, m\}$.
For a path $p$, set $\overbar{p} := \bar{\tau}(p)$.

\subsection{Polynomial rings are geodesic ghor algebras} \label{polynomialsection}
\ \\
\indent The simplest possible geodesic ghor algebra on a $2N$-gon $P$ has one arrow along each side, and a single diagonal arrow in the interior of $P$.  
The cases $N = 2,3,4$ are shown in Figure~\ref{fig:polynomial}.

In such a quiver, each arrow is contained in a unique simple matching.
Thus there is one simple matching for each arrow.
Since each arrow is a cycle, the center, cycle algebra, and ghor algebra all coincide.
Consequently, the ghor algebra is the polynomial ring in $N+1$ variables.  
It follows that every polynomial ring in at least $3$ variables arises as a ghor algebra.

\begin{figure}
\includegraphics[width=10cm]{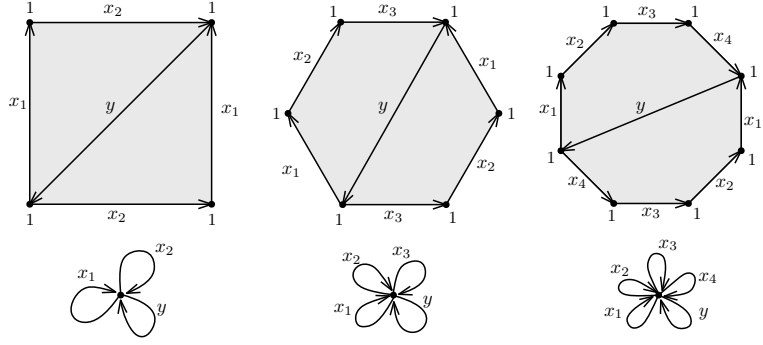}
\caption{The polynomial ghor algebras $k[x_1,x_2,y]$, $k[x_1,x_2,x_3,y]$, and $k[x_1,x_2,x_3,x_4,y]$.}
\label{fig:polynomial}
\end{figure}

\subsection{A geodesic ghor algebra on a genus $2$ surface} \label{octagonsection}
\ \\
\indent Consider the ghor algebra $A = kQ/\ker \eta$ on a smooth genus $2$ surface with quiver $Q$ given in Figure~\ref{fig:octagon}.i (the identifications of the sides of the polygon $P$ are indicated by color). 
We will determine the center $R$ and cycle algebra $S$ of $A$ explicitly, and show that $R$ is nonnoetherian of Krull dimension $5$.

$Q$ has twelve simple matchings, shown in Figure \ref{simple matchings octagon}.
The nontrivial simple matchings $y_j$ and $z_j$ may be determined using special partitions of $Q_1$ called subdivisions, which are described in \cite[Section 3]{BB}.
The polynomial ring $k[\mathcal{S}]$ is thus
$$k[\mathcal{S}] = k[x_1, \ldots, x_4, y_1, \ldots, y_4, z_1, \ldots, z_4].$$
Note that $x_{j+1}$ is the rotation of $x_j$ by $\frac{\pi}{2}$ in the counter-clockwise direction, and similarly for the simple matchings $y_j$ and $z_j$.
Using the four $x_j$ simple matchings and Lemma \ref{geo}, it is straightforward to verify that $A$ is geodesic.

Consider the eight cycles $\alpha_1, \ldots, \alpha_8$ that run along the sides of the fundamental polygon $P$, as shown in Figure \ref{fig:octagon}.ii.
These cycles have $\bar{\tau}$-images
$$\overbar{\alpha}_j = \left\{ \begin{array}{ll}
x_{j+1}x_{j+2}^2x_{j+3} y_j y_{j+2}y_{j+3}^2z_j z_{j+1}^2z_{j+2} & \text{ if $j$ is odd}\\
(x_{j+2}x_{j+3}y_{j+2}y_{j+3}z_{j+2}z_{j+3})^2 & \text{ if $j$ is even}
\end{array} \right.$$
Observe that $A$ has cycle algebra
$$S = k[\sigma, \overbar{\alpha}_j \, | \, j \in [8]] \subset k[\mathcal{S}],$$
and center
$$R = k[\sigma] + (\overbar{\alpha}_j \overbar{\alpha}_{j+1} \overbar{\alpha}_{j+2}, \sigma^2 \, | \, j \in [8])S,$$
with indices taken modulo $8$.

The center $R$ is nonnoetherian since it contains the infinite ascending chain of ideals
$$\overbar{\alpha}_1\sigma^2R \subset (\overbar{\alpha}_1, \overbar{\alpha}_1^2)\sigma^2R \subset (\overbar{\alpha}_1, \overbar{\alpha}_1^2, \overbar{\alpha}_1^3) \sigma^2R \subset \cdots.$$
Nevertheless, we claim that $R$ satisfies the ascending chain condition on prime ideals, and in particular has Krull dimension $5 = 2g +1$.

If $\mathfrak{n} \in \operatorname{Max}S$ is a point for which none of the monomial generators of $S$ vanish (for example, if each monomial is set equal to $1$), then $R_{\mathfrak{n} \cap R} = S_{\mathfrak{n}}$.
Consequently, the noetherian locus $U_{S/R}$ is nonempty.
Therefore $\operatorname{dim}R = \operatorname{dim}S$, by \cite[Theorem 2.5.4]{B3}.
It thus suffices to show that $S$ has Krull dimension $5$.

For $j \in [5]$ set
$$\mathfrak{p}_j := (\sigma, \overbar{\alpha}_1, \ldots, \overbar{\alpha}_{3 + j})S,$$
and consider the chain of ideals of $S$
\begin{equation} \label{p0}
0 \subseteq \mathfrak{p}_1 \subseteq \mathfrak{p}_2 \subseteq \cdots \subseteq \mathfrak{p}_5.
\end{equation}

By \cite[Theorem 3.11]{BB}, two cycles $p$ and $q$ are in the same class if and only if $\overbar{p} = \overbar{q} \sigma^{\ell}$ for some $\ell \in \mathbb{Z}$.
In particular, if $\Sigma$ is smooth, then two cycles are homologous if and only if $\overbar{p} = \overbar{q}$.
Furthermore, suppose $\mathfrak{q}$ is a prime ideal of $S$.
Then $\overbar{\alpha}_i \in \mathfrak{q}$ implies $\sigma \in \mathfrak{q}$ since $\overbar{\alpha}_i \overbar{\alpha}_{i+4} = \sigma^2$.
But this then implies $\overbar{\alpha}_k$ or $\overbar{\alpha}_{k+4}$ is in $\mathfrak{q}$ for each $k \in [4]$.
Therefore each $\mathfrak{p}_j$ is prime, and $\mathfrak{p}_1$ is a height one prime.
Consequently, $\operatorname{dim}S \leq 5$.
But the inclusions in (\ref{p0}) are strict, again since any two cycles $p$ and $q$ are homologous if and only if $\overbar{p} = \overbar{q} \sigma^{\ell}$ for some $\ell \in \mathbb{Z}$.
Whence $\operatorname{dim}S \geq 5$.
It follows that $\operatorname{dim}R = \operatorname{dim}S = 5$.

\begin{figure}
$$\begin{array}{ccc}
\includegraphics[scale=.63]{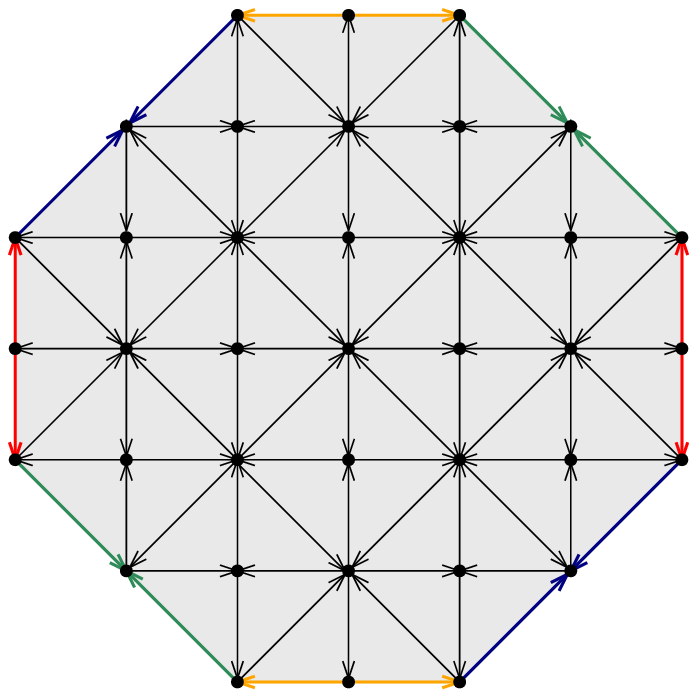}
&  
\includegraphics[scale=.63]{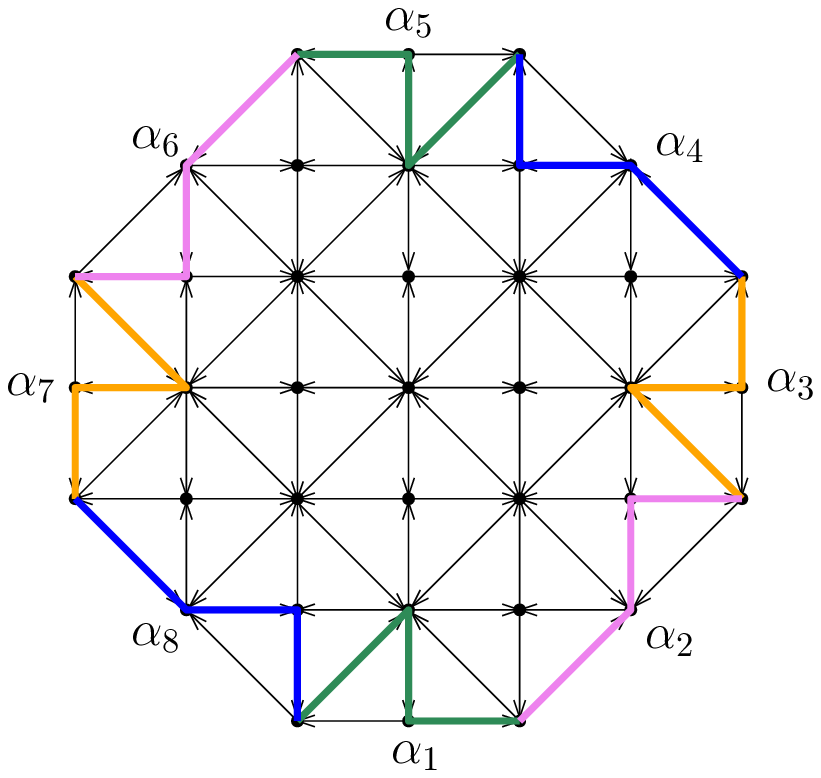}
&
\includegraphics[scale=.63]{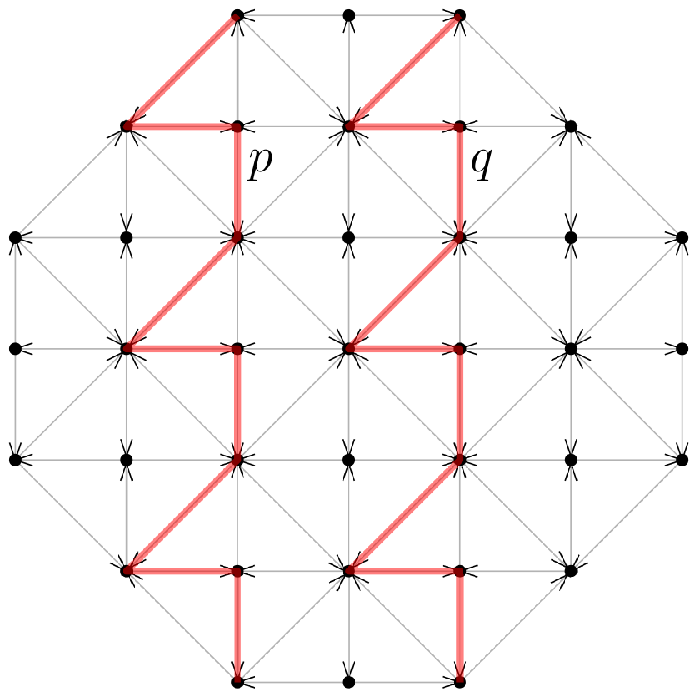}
\\
(i) & (ii) & (iii)
\end{array}$$
\caption{(i) The dimer quiver $Q$ of Section \ref{octagonsection}.  
(ii) The cycles $\alpha_j$ run along the sides of the fundamental polygon $P$.
(iii) The homologous cycles $p,q$ are equal in the ghor algebra $A$, but not in the dimer algebra of $Q$.} 
\label{fig:octagon}
\end{figure}

\begin{figure}
\includegraphics[width=14cm]{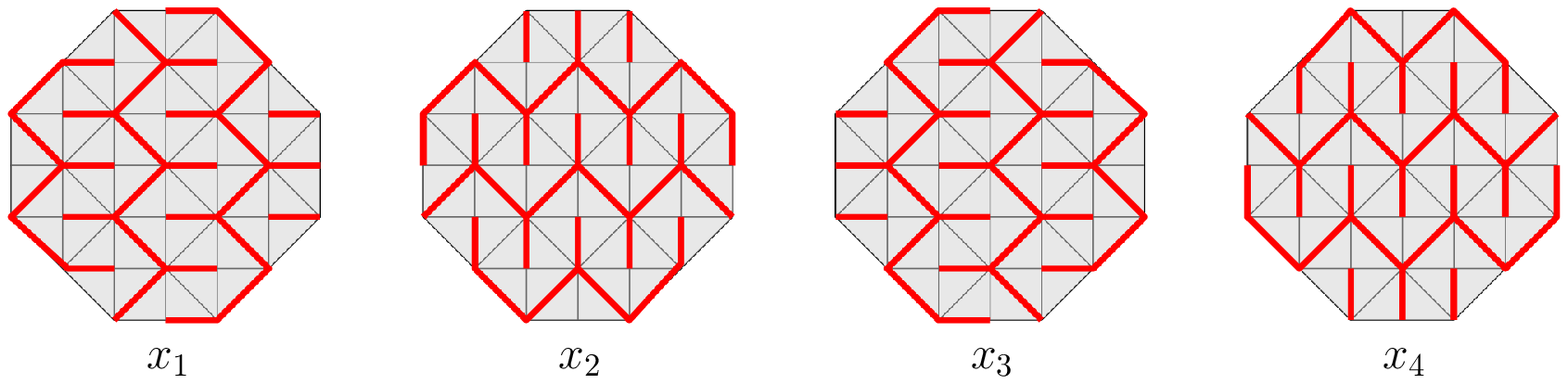}\\
\ \\
\includegraphics[width=14cm]{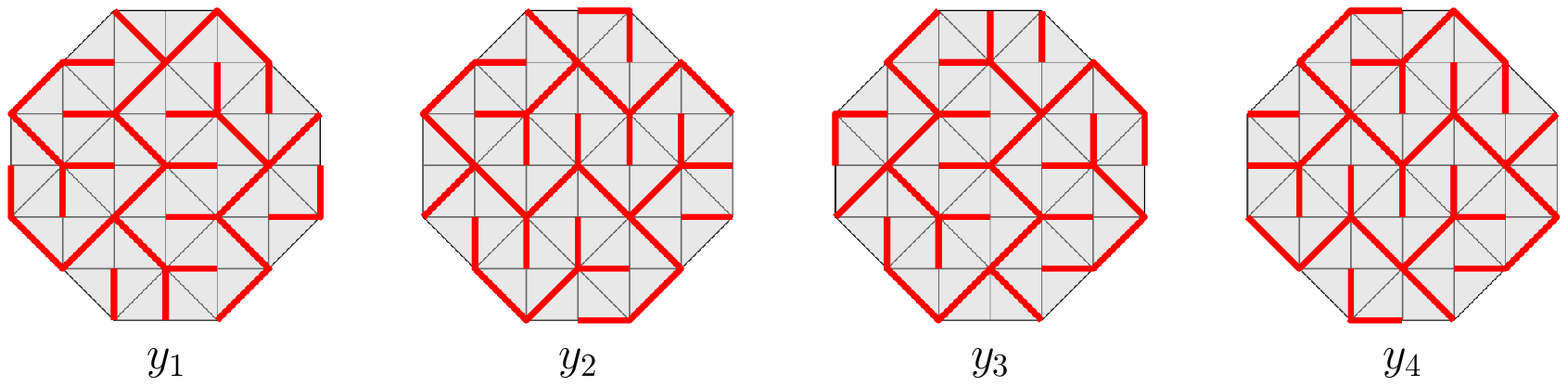}\\ 
\ \\
\includegraphics[width=14cm]{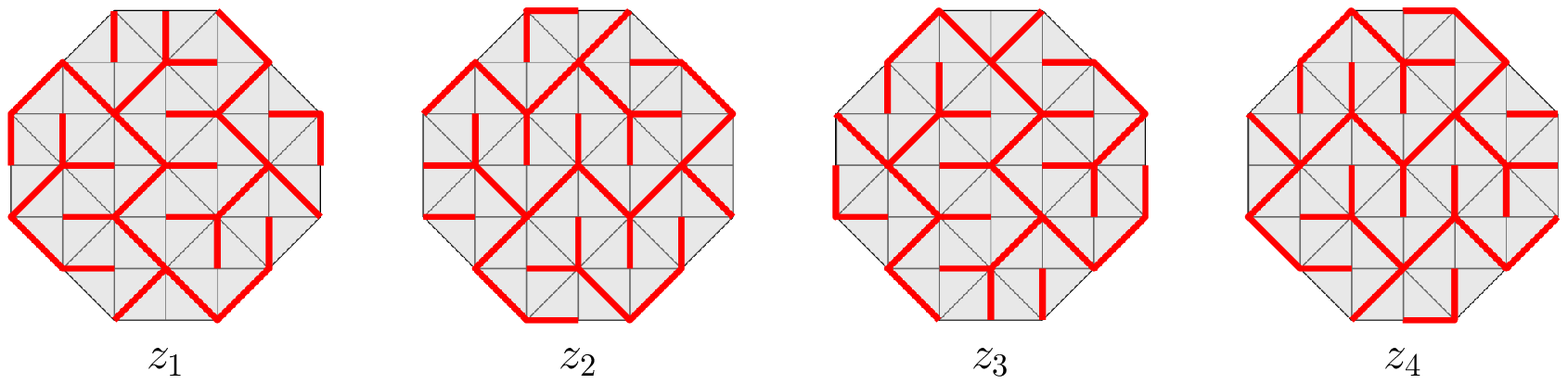}
\caption{The simple matchings for the geodesic ghor algebra of Section \ref{octagonsection}.}
\label{simple matchings octagon}
\end{figure}

\subsection{Flower of life: a geodesic ghor algebra on a pinched torus} \label{folsection}
\ \\
\indent Finally, consider the ghor algebra $A = kQ/\ker \eta$ on a pinched torus with quiver $Q$ given in
Figure~\ref{fig:flower}.i
(the identifications of the sides of the polygon $P$ are indicated by color).\footnote{The quiver in this example fits into a pattern of overlapping circles called the `flower of life' in the New Age literature.} 
We will determine the center $R$ and cycle algebra $S$ of $A$ explicitly, and show that $R$ is nonnoetherian of Krull dimension $4$.

\begin{figure}
$$\begin{array}{ccc}
\includegraphics[scale=.725]{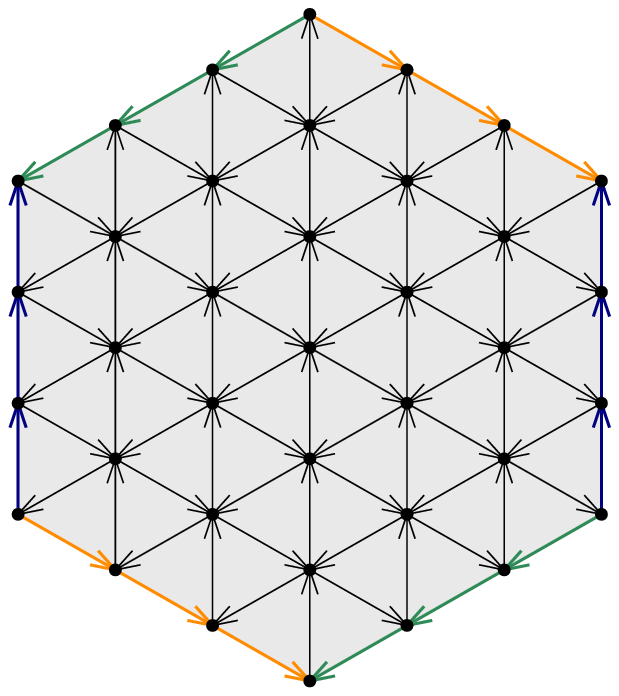}
&
\includegraphics[scale=.725]{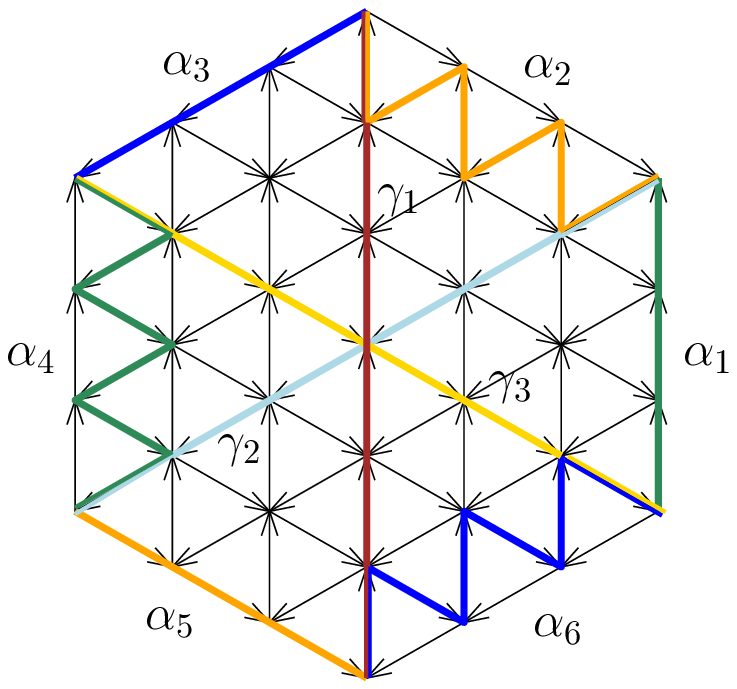}
&
\includegraphics[scale=.725]{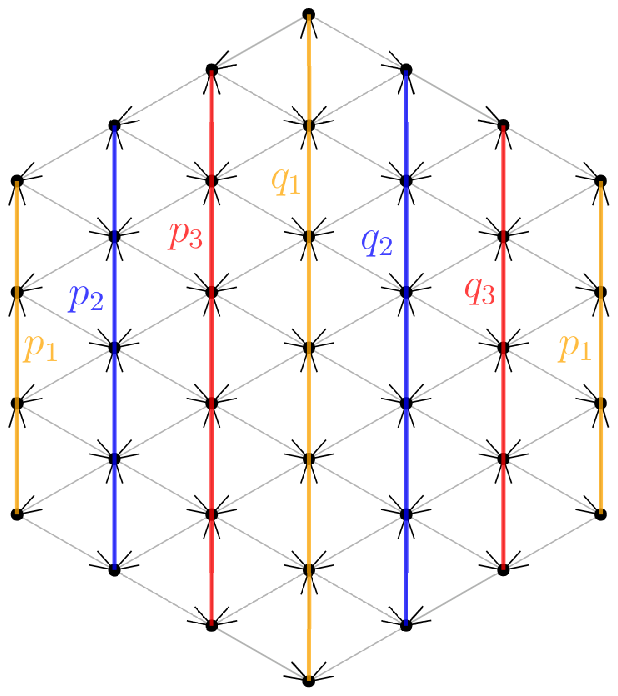}
\\
(i) & (ii) & (iii)
\end{array}$$
\caption{(i) The dimer quiver $Q$ of Section \ref{folsection}.
(ii) The cycles $\alpha_j$ run along the sides of the fundamental polygon $P$.
(iii) The cycles $q_jp_j$ each have $\bar{\tau}$-image $\overbar{\gamma_1 \alpha_1}$ (note that $p_1 = \alpha_1$ and $q_1 = \gamma_1$).}
\label{fig:flower}
\end{figure}

\begin{figure}
$$\begin{array}{ccccc}
\includegraphics[scale=.725]{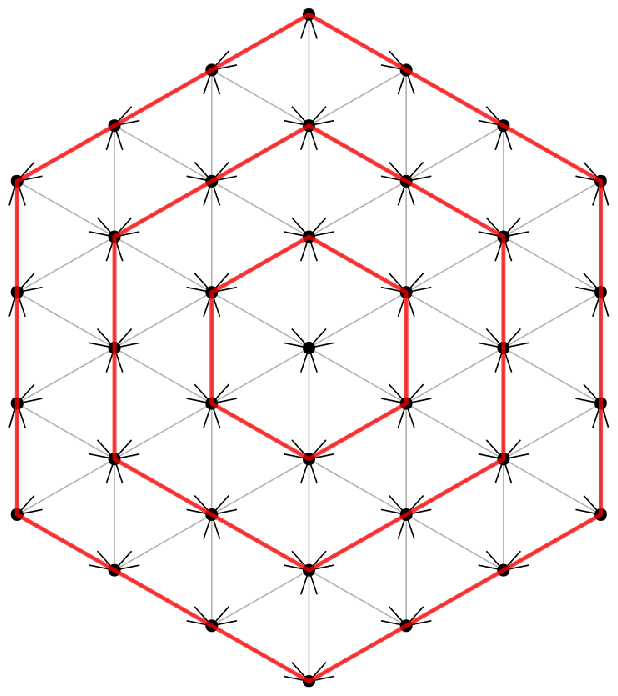}
& \ &
\includegraphics[scale=.725]{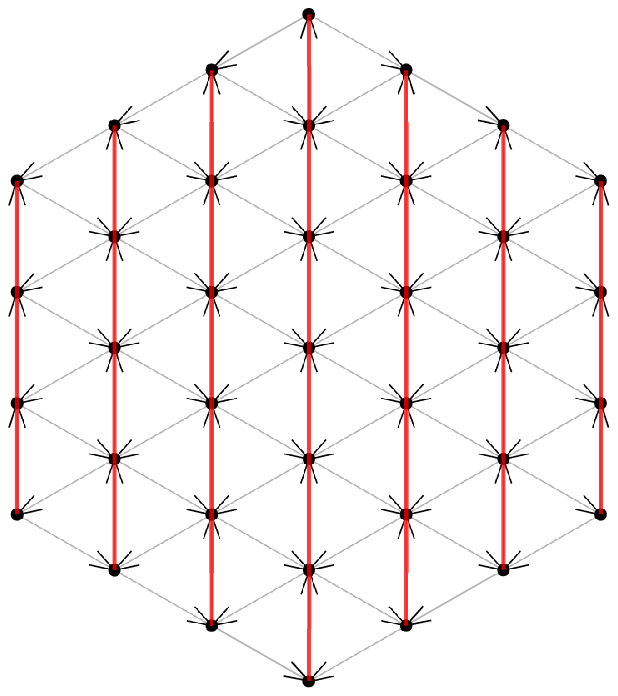}
& \ &
\includegraphics[scale=.725]{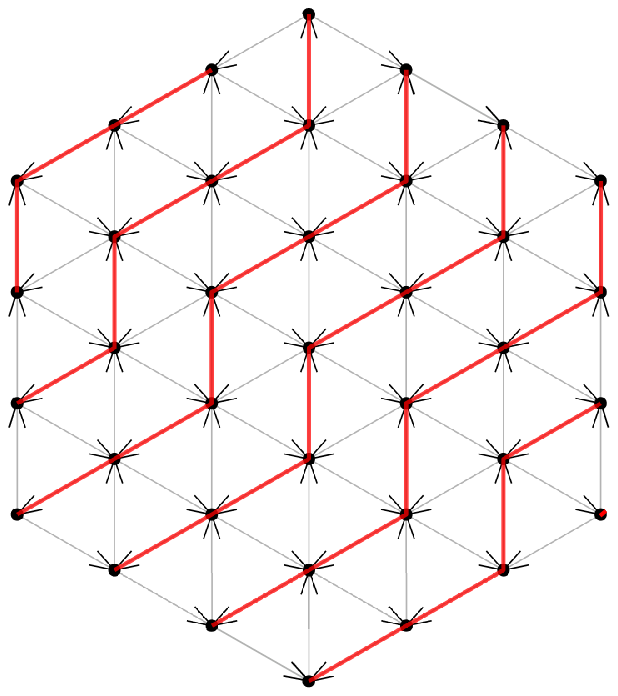}
\end{array}$$
\caption{The three types of the ten simple matchings for the geodesic ghor algebra of Section \ref{folsection}.}
\label{fig:flower-matchings}
\end{figure}

$Q$ has ten simple matchings, three of which are shown in Figure~\ref{fig:flower-matchings}.
Specifically, $Q$ has one simple matching consisting of concentric circles, three simple matchings for the three straight directions of $Q$ ($Q$ is symmetric upon rotations by $\frac{2 \pi}{3}$ and $\frac{4 \pi}{3}$), and six simple matchings that are `wiggly' (and identical up to rotation by a multiple of $\frac{\pi}{3}$).
Using these simple matchings and Lemma \ref{geo}, it is straightforward to verify that $A$ is geodesic.

As shown in Figure \ref{fig:flower}.ii, denote by $\alpha_j$, $j \in [6]$, the cycles that run along the sides of the fundamental polygon $P$; by $\gamma_k$, $k \in [3]$, the `straight' cycles that pass through the center of $P$; and by $\delta_k$, $k \in [3]$, the cycles that run opposite to $\gamma_k$ and also pass through the center of $P$.
We claim that $A$ has cycle algebra
\begin{equation} \label{cycle ex}
S = k[\sigma, \overbar{\alpha}_j, \overbar{\gamma}_k \, | \, j \in [6], \, k \in [3]] \subset k[\mathcal{S}].
\end{equation}
Indeed, consider the set of all vertical paths $p_j,q_j$, $j \in [3]$, shown in Figure \ref{fig:flower}.iii.
The concatenations $q_jp_j$ are cycles in $Q$.
These cycles satisfy
$$\overbar{q_jp_j} = \overbar{\gamma_1\alpha_1}, \ \ \ j \in [3].$$
Furthermore, the cycle $\delta_1$ satisfies
$$\overbar{\delta}_1 = \overbar{\alpha_3 \alpha_4 \alpha_5}.$$
Similar equalities hold for the southeasterly and southwesterly directions.
Thus the ten monomials given in (\ref{cycle ex}) generate $S$.

Observe that $A$ has center
$$R = k[\sigma] + (\overbar{\alpha}_{2j}(\overbar{\alpha}_{2j+1}, \overbar{\alpha}_{2j+2}, \overbar{\alpha}_{2j+3}, \sigma) \, | \, j \in [3])S,$$
with indices taken modulo $6$.
$R$ is nonnoetherian since it contains the infinite ascending chain of ideals
$$\overbar{\alpha}_1\sigma^3R \subset (\overbar{\alpha}_1, \overbar{\alpha}_1^2)\sigma^3R \subset (\overbar{\alpha}_1, \overbar{\alpha}_1^2, \overbar{\alpha}_1^3) \sigma^3R \subset \cdots.$$

To show that $R$ has Krull dimension $4$, it suffices to show that $S$ has Krull dimension $4$, as in Section \ref{octagonsection}. 
For $j \in [6]$ and $k \in [3]$, we have the relations
\begin{equation*} \label{sigma3}
\overbar{\alpha}_j\overbar{\alpha}_{j+3} = \sigma^3 \ \ \ \text{ and } \ \ \ \overbar{\gamma}_k \overbar{\delta}_k = \sigma^6.
\end{equation*}
We also have the three relations
$$\overbar{\gamma}_1 \overbar{\alpha}_4 = \overbar{\alpha}_6 \overbar{\alpha}_2, \ \ \ \ \overbar{\gamma}_2 \overbar{\alpha}_6 = \overbar{\alpha}_2 \overbar{\alpha}_4, \ \ \ \ \overbar{\gamma}_3 \overbar{\alpha}_3 = \overbar{\alpha}_4 \overbar{\alpha}_6,$$
and the six homotopy relations
$$\overbar{\gamma}_1\overbar{\alpha}_5 = \overbar{\alpha}_6 \overbar{\alpha}_1, \ \ \ \ \overbar{\gamma}_1 \overbar{\alpha}_3 = \overbar{\alpha}_2 \overbar{\alpha}_1, \ \ \ \ \ldots, \ \ \ \ \overbar{\gamma}_3 \overbar{\alpha}_3 = \overbar{\alpha}_4 \overbar{\alpha}_5, \ \ \ \ \overbar{\gamma}_3 \overbar{\alpha}_1 = \overbar{\alpha}_6 \overbar{\alpha}_5.$$
Thus, following the arguments of Section \ref{octagonsection}, the chain of ideals of $S$
$$\begin{array}{ccl}
0 & \subset & (\sigma, \overbar{\alpha}_1, \overbar{\alpha}_2, \overbar{\alpha}_3, \overbar{\alpha}_4, \overbar{\alpha}_5, \overbar{\gamma}_2)S =: \mathfrak{p}_1 \\
& \subset & (\mathfrak{p}_1, \overbar{\alpha}_6)S\\
& \subset & (\mathfrak{p}_1, \overbar{\alpha}_6, \overbar{\gamma}_1)S\\
& \subset & (\mathfrak{p}_1, \overbar{\alpha}_6, \overbar{\gamma}_1, \overbar{\gamma}_3)S
\end{array}$$
is a maximal chain of primes.
Therefore $\operatorname{dim}R = \operatorname{dim}S = 4$.

\ \\
\textbf{Acknowledgments.}
The authors were supported by the Austrian Science Fund (FWF) grant P 30549-N26.
The first author was also supported by the Austrian Science Fund (FWF) grant W 1230 and by a Royal Society Wolfson Fellowship.

\bibliographystyle{hep}
\def\cprime{$'$} \def\cprime{$'$}

\end{document}